\documentclass[english,11pt]{amsart}
\usepackage[T1]{fontenc}
\usepackage[utf8]{inputenc}
\usepackage[dvipsnames]{xcolor}
\usepackage{babel}
\usepackage{prettyref}
\usepackage{amsthm}
\usepackage{amsmath}
\usepackage{amssymb}
\usepackage{mathrsfs}
\usepackage{mathtools}
\usepackage{graphicx}
\usepackage[percent]{overpic}
\usepackage[unicode=true,pdfusetitle,bookmarks=true,bookmarksnumbered=false,bookmarksopen=false,breaklinks=false,pdfborder={0 0 0},backref=false,colorlinks=false]{hyperref}
\usepackage{csquotes}
\usepackage{enumitem}
\usepackage{caption}
\usepackage{subcaption}
\usepackage{changepage}

\theoremstyle{plain}
\newtheorem{theorem}{\protect\theoremname}
\newtheorem{proposition}[theorem]{\protect\propositionname}
\newtheorem{lemma}[theorem]{\protect\lemmaname}

\theoremstyle{remark}

\theoremstyle{definition}
\newtheorem{definition}[theorem]{\protect\definitionname}

\numberwithin{theorem}{section}

\providecommand{\corollaryname}{Corollary}
\providecommand{\definitionname}{Definition}
\providecommand{\examplename}{Example}
\providecommand{\lemmaname}{Lemma}
\providecommand{\notationname}{Notation}
\providecommand{\remarkname}{Remark}
\providecommand{\remarkname}{Definition}
\providecommand{\propositionname}{Proposition}
\providecommand{\theoremname}{Theorem}
\providecommand{\equationname}{Equation}
\providecommand{\figurename}{Figure}
\providecommand{\casename}{Case}

\providecommand{\sectionname}{Section}

\newrefformat{def}{\definitionname~\ref{#1}}
\newrefformat{cor}{\corollaryname~\ref{#1}}
\newrefformat{thm}{\theoremname~\ref{#1}}
\newrefformat{prop}{\propositionname~\ref{#1}}
\newrefformat{lem}{\lemmaname~\ref{#1}}
\newrefformat{ex}{\examplename~\ref{#1}}
\newrefformat{sub}{\sectionname~\ref{#1}}
\newrefformat{eq}{\equationname~\eqref{#1}}
\newrefformat{fig}{\figurename~\ref*{#1}}
\newrefformat{case}{\casename~\ref*{#1}}

\newcommand{\R}{\mathbb R}
\newcommand{\trappedregion}{\mathcal T}
\renewcommand{\div}{\operatorname{div}}
\newcommand{\grad}{\operatorname{grad}}
\newcommand{\Hm}[1]{\mathcal H^{#1}}

\newcommand{\reg}{\operatorname{reg}}
\newcommand{\sing}{\operatorname{sing}}
\newcommand{\rpartial}{\partial^*}

\newcounter{casenumber}

\newcommand{\casedef}[2]{\protect\refstepcounter{casenumber}\textbf{Case~\Roman{casenumber}:} #2\label{#1}}

\newenvironment{proofcase}[2]{
\vspace{.2pc}
\noindent \textbf{Case~#1: #2}
\begin{adjustwidth}{2pc}{}
}
{
\end{adjustwidth}
\vspace{.3pc}
}

\title[Regularity of the boundary of the trapped region]{Regularity of the boundary of the trapped region in asymptotically Euclidean Riemannian manifolds of arbitrarily large dimensions}
\author{Eric Larsson}

\address{Department of mathematics \\
  KTH \\
  SE-100 44 Stockholm \\
  Sweden} 
\email[Eric Larsson]{ericlar@kth.se}

\begin{document}
\maketitle

\begin{abstract}
We prove that the boundary of the trapped region in an asymptotically Euclidean Riemannian manifold of dimension at least \(3\) is a stable smooth minimal hypersurface except for a singular set of codimension at least 8.
\end{abstract}

\setcounter{tocdepth}{1}
\tableofcontents

\section{Introduction}
Asymptotically Euclidean initial data slices for Einstein's equation are used to model isolated black holes.
One way of thinking about black holes in an initial data setting is to consider the \enquote{trapped region}.
The conventional way of defining the trapped region in \(3\)-dimensional initial data slices is by using the concept of \enquote{outer expansion} to define \enquote{weakly outer trapped surfaces}.
For details, see \cite{AnderssonMetzger09} \cite{Eichmair10} \cite{AnderssonEichmairMetzger11}.
The outer expansion of a smooth hypersurface is an analog of the mean curvature in the purely Riemannian setting, and it is defined by modifying the mean curvature by a term from the second fundamental form of the initial data slice.
A compact hypersurface is weakly outer trapped if its outer expansion is everywhere nonpositive.
The trapped region is then the union of all sets which are enclosed by weakly outer trapped surfaces.
From this definition, there is no immediate reason that the boundary of the trapped region, which can be interpreted as an initial-data version of the boundary of the black hole, should have any particular regularity.
However, it was proved in work by Andersson, Eichmair, and Metzger \cite{AnderssonMetzger09} \cite{Eichmair10} \cite{AnderssonEichmairMetzger11} under the condition that the initial data slice has dimension \(3 \leq n \leq 7\), that the boundary of the trapped region is a smooth hypersurface, that it is \enquote{marginally outer trapped} in the sense that its outer expansion is everywhere zero, and that it is stable.
A natural question is what can be said about the regularity of the boundary of the trapped region in higher dimensions.
Of course, this depends on how the trapped region in a high-dimensional manifold is defined.
We are not aware of any standard definition, but we discuss our choice briefly in \prettyref{sec:prehorizon-domains-and-the-trapped region}.
The expected answer to the high-dimensional regularity question is that the boundary of the trapped region is marginally trapped, stable, and smooth outside of a singular set of codimension at least \(8\).
This expectation is supported by an analogy with the theory of minimal hypersurfaces.
In the case when the second fundamental form of the initial data slice is identically zero, the initial data slice is simply an asymptotically Euclidean Riemannian manifold, and the outer expansion is simply the mean curvature.
In this case, the results in dimensions less than \(8\) say that the boundary of the trapped region is the outermost minimal hypersurface, and that it is stable.
It is known from work by Schoen and Simon \cite{SchoenSimon81} that stable minimal hypersurfaces are smooth outside of a singular set of codimension at least \(8\).
One might then hope that this regularity holds for the boundary of the trapped region in all dimensions.

The result of this paper is that this expected regularity indeed holds in all dimensions in the case when the second fundamental form of the initial data slice is identically zero.

\subsection*{Acknowledgements}
I would like to thank Mattias Dahl and Hans Ringström for helpful suggestions concerning this paper.

\section{Sets of locally finite perimeter}
There are several ways of working with nonsmooth analogues of minimal hypersurfaces.
We have chosen to primarily use sets of locally finite perimeter, for which \cite{Giusti84} is a good reference.
The minimal hypersurfaces we are interested in are boundaries, and the approach using sets of locally finite perimeter reflects the fact that the sets, and not only their boundaries, are of interest.

\begin{definition}
A subset \(E\) of a Riemannian manifold \((M, g)\) with boundary is a \emph{set of locally finite perimeter} if for every open set \(\Omega \subseteq M\) with compact closure it holds that \(P(E, \Omega) < \infty\), where
\[
P(E, \Omega)
=
\sup_{\substack{X \in \mathfrak{X}_0^{C^1}(\Omega) \\ ||X||_g \leq 1 \text{ everywhere}}} \int_E \div X \, d\Hm{n}.
\]
Here \(\mathfrak{X}_0^{C^1}(\Omega)\) denotes the set of \(C^1\) vector fields which are compactly supported in \(\Omega\), and \(\Hm{n}\) denotes \(n\)-dimensional Hausdorff measure on \((M, g)\).
The quantity \(P(E, \Omega)\) is called the \emph{perimeter} of \(E\) in \(\Omega\).
\end{definition}

Following \cite[Definition~3.3]{Giusti84}, we use \(\rpartial E\) to denote the reduced boundary of a set \(E\) of locally finite perimeter.
Note that a set has locally finite perimeter if and only if its image under any coordinate chart has locally finite perimeter in the Euclidean metric on the coordinate chart.
This means that the compactness theorem for sets of locally finite perimeter in Euclidean space \cite[Theorem~1.19]{Giusti84} carries over to Riemannian manifolds.
\begin{lemma}\label{lem:Caccioppoli-compactness}
Let \((E_i)_{i = 1}^\infty\) be a sequence of sets of locally finite perimeter in a Riemannian manifold, and suppose that their perimeters are uniformly bounded.
Then there is a subsequence which converges (in \(L^1_{\mathrm{loc}}\)-norm for the indicator functions) to a set of locally finite perimeter.
\end{lemma}
Similarly, the theorem about semicontinuity of perimeter \cite[Theorem~1.9]{Giusti84} also carries over.
\begin{lemma}\label{lem:perimeter-semicontinuity}
Let \((E_i)_{i = 1}^\infty\) be a sequence of sets of locally finite perimeter in a Riemannian manifold which converge in \(L^1\)-norm for the indicator functions to a set \(E\) of locally finite perimeter.
Then it holds for every open set \(\Omega\) that
\[
P(E, \Omega)
\leq
\liminf_{i \to \infty} P(E_i, \Omega).
\]
\end{lemma}

\section{Stationarity over singular sets}
In the proof of \prettyref{prop:prehorizon-set-union}, we will need to prove that a certain set is stationary.
However, the methods used in that proof only prove stationarity with respect to variations which are compactly supported outside of a small set.
The purpose of this section is to prove \prettyref{lem:extension-of-stationarity}, which will allow us to conclude that the set is stationary with respect to all variations.

\begin{definition}
Let \(M\) be a smooth manifold and let \(S\) be a subset of \(M\).
We define \(\reg S\) to be the set of points \(x \in \overline S\) such that there is an open neighborhood \(U \subset M\) of \(x\) such that \(U \cap \overline S\) is a connected \(C^2\) hypersurface without boundary.
We define \(\sing S = \overline S \setminus \reg S\).
Note that \(\sing S\) is a closed set.
\end{definition}

\begin{lemma}\label{lem:extension-of-stationarity}
Let \((M, g)\) be a Riemannian manifold of dimension \(n \geq 3\).
Let \(S \subseteq M\) be a set which satisfies \(\Hm{n-1}(\sing S) = 0\).
Suppose that there is a constant \(\omega\) such that \(\Hm{n-1}(S \cap B(x, r)) < \omega r^{n-1}\) for all \(x \in M\) and \(r > 0\), where \(B(x, r)\) is the \(n\)-dimensional open ball of radius \(r\) around \(x\) in \(M\).
If \(S\) is stationary with respect to all variations which are compactly supported outside of a compact set \(A\) with \(\Hm{n-2}(A) = 0\), then \(S\) is stationary with respect to all compactly supported variations.
\end{lemma}
\begin{proof}
Since \(\Hm{n-1}(\sing S) = 0\), it holds that \(S\) is stationary if and only if \(\reg S\) is stationary.
We may without loss of generality assume that \(S = \reg S\), and we will do so for notational convenience.
We need to prove that
\[
\int_{S} \div_S(X) \, d\Hm{n-1}
=
0
\]
for all compactly supported \(C^1\) vector fields \(X\) on \(M\), where \(\div_{S}(X)\) is the divergence of \(X\) along the hypersurface \(S\).
Fix such a vector field \(X\).

Since \(S\) is stationary with respect to variations which are compactly supported in \(M \setminus A\), it holds that
\[
\int_{S} \div_{S}(\eta X) \, d\Hm{n-1}
=
0\]
if \(\eta\) is a smooth function which is zero on a neighborhood of \(A\).
Let \(\pi\) denote the orthogonal projection of \(TM\) onto \(TS\).
Since
\[
\div_{S}(\eta X)
=
\eta \div_{S}(X) + \pi(X)(\eta)
\]
we have
\[\begin{aligned}
&\left| \int_{S} \eta \div_{S}(X) \, d\Hm{n-1} \right|
\\
&=
\left| \int_{S} \div_{S}(\eta X) \, d\Hm{n-1} - \int_{S} \pi(X)(\eta) \, d\Hm{n-1} \right|
\\
&=
\left| \int_{S} \pi(X)(\eta) \, d\Hm{n-1} \right|
\\
&\leq
\left( \sup_M |X| \right) \int_{S} |d\eta| \, d\Hm{n-1}.
\end{aligned}\]
We will now construct a family of functions \((\eta_\epsilon)_{\epsilon > 0}\) such that
\[
\lim_{\epsilon \to 0} \int_{S} \eta_\epsilon \div_{S}(X) \, d\Hm{n-1}
=
\int_{S} \div_{S}(X) \, d\Hm{n-1}
\]
and
\[
\int_{S} |d\eta_\epsilon| \, d\Hm{n-1}
<
\epsilon,
\]
thereby proving that
\[
\int_{S} \div_{S}(X) \, d\Hm{n-1}
=
0.
\]

Since \(\Hm{n-2}(A) = 0\), there is for every \(\epsilon > 0\) a cover of \(A\) by open balls
\[
A
\subseteq
\bigcup_{i = 1}^\infty B(x_{\epsilon, i}, r_{\epsilon, i})
\]
where
\[
\sum_{i = 1}^\infty r_{\epsilon, i}^{n-2}
<
\frac{\epsilon}{2^n \omega},
\]
and since \(A\) is compact, there is a finite subcover
\[
A
\subseteq
\bigcup_{i = 1}^{N_\epsilon} B(x_{\epsilon, i}, r_{\epsilon, i})
\]
where
\[
\sum_{i = 1}^{N_\epsilon} r_{\epsilon, i}^{n-2}
<
\frac{\epsilon}{2^n \omega}.
\]
For each \(i\), let \(\eta_{\epsilon, i}\) be a smooth function such that
\[
\eta_{\epsilon, i}(x)
=
\begin{cases}
	1 \text{ if } x \in B(x_{\epsilon, i}, r_{\epsilon, i}),\\
	0 \text{ if } x \notin B(x_{\epsilon, i}, 2r_{\epsilon, i})
\end{cases}
\]
and \(||d\eta_{\epsilon, i}||_g < 2/r_{\epsilon, i}\) everywhere.
Let
\[
\eta_\epsilon(x)
=
1 - \prod_{i = 1}^{N_\epsilon} \eta_{\epsilon, i}(x).
\]
By the dominated convergence theorem,
\[
\lim_{\epsilon \to 0} \int_{S} \eta_\epsilon \div_{S}(X) \, d\Hm{n-1}
=
\int_{S} \div_{S}(X) \, d\Hm{n-1}.
\]
It holds that
{\allowdisplaybreaks
\begin{align*}
\int_{S} |d\eta_\epsilon| \, d\Hm{n-1}
&=
\int_{S} \left| \sum_{i = 1}^{N_\epsilon} \left(\prod_{j \neq i} \eta_{\epsilon, j}\right) d\eta_{\epsilon, i} \right| \, d\Hm{n-1}
\\
&\leq
\sum_{i = 1}^{N_\epsilon} \int_{S} \left|\prod_{j \neq i} \eta_{\epsilon, j}\right| \left|d\eta_{\epsilon, i} \right| \, d\Hm{n-1}
\\
&\leq
\sum_{i = 1}^{N_\epsilon} \int_{S} \left|d\eta_{\epsilon, i} \right| \, d\Hm{n-1}
\\
&=
\sum_{i = 1}^{N_\epsilon} \int_{S \cap B(x_{\epsilon, i}, 2r_{\epsilon, i})} \left|d\eta_{\epsilon, i} \right| \, d\Hm{n-1}
\\
&\leq
\sum_{i = 1}^{N_\epsilon} \int_{S \cap B(x_{\epsilon, i}, 2r_{\epsilon, i})} \frac{2}{r_{\epsilon, i}} \, d\Hm{n-1}
\\
&=
\sum_{i = 1}^{N_\epsilon}  \frac{2}{r_{\epsilon, i}} \Hm{n-1}(S \cap B(x_{\epsilon, i}, 2r_{\epsilon, i}))
\\
&\leq
2^n \omega \sum_{i = 1}^{N_\epsilon} r_{\epsilon, i}^{n-2}
\\
&\leq
\epsilon.
\end{align*}
}
Hence \(S\) is stationary with respect to all compactly supported variations.
\end{proof}

\section{Prehorizon domains and the trapped region}\label{sec:prehorizon-domains-and-the-trapped region}
This section defines the trapped region of an asymptotically Euclidean manifold in a way which is useful for working with manifolds of dimension greater than \(7\).
In low dimensions, it is natural to define the trapped region as the union of all sets which are bounded by smooth hypersurfaces with nonpositive mean curvature.
Since the solution of an area minimization problem in the region outside of a hypersurface with nonpositive mean curvature gives a smooth minimal hypersurface, this definition is equivalent to defining the trapped region as the union of all sets bounded by smooth minimal hypersurfaces.
In higher dimensions, the solutions of area minimization problems are not guaranteed to be smooth.
If we define the trapped region using smooth minimal hypersurfaces, then area minimization problems are not immediately applicable as a tool for determining its properties.
Instead, we propose to define the trapped region as the union of \enquote{prehorizon domains}, as defined below.
If the manifold has dimension at most \(7\), then this definition agrees with the definitions in terms of smooth hypersurfaces and smooth hypersurfaces with nonpositive mean curvature.
In higher dimensions, this trapped region includes the union of all sets bounded by smooth hypersurfaces with nonpositive mean curvature, but we do not know whether they are necessarily equal.

\begin{definition}
Let \((M, g)\) be an asymptotically Euclidean Riemannian manifold.
A set \(E \subseteq M\) of locally finite perimeter is a \emph{bounded domain} (with respect to the chosen asymptotically Euclidean end) if
\begin{itemize}
\item \(E\) is open,
\item the complement of \(E\) is a neighborhood of the chosen asymptotically Euclidean end,
\item after compactifying the chosen asymptotically Euclidean end, the complement of \(E\) is compact,
\item \(\partial E = \overline{\rpartial E}\).
\end{itemize}
\end{definition}

\begin{definition}
Let \((M, g)\) be an asymptotically Euclidean Riemannian manifold.
A bounded domain \(E \subseteq M\) is \emph{outer area minimizing} (with respect to the chosen asymptotically Euclidean end) if \(P(E, M) < \infty\) and there is no bounded domain \(E' \supset E\) such that \(P(E', M) < P(E, M)\).
\end{definition}

\begin{definition}
Let \((M, g)\) be an \(n\)-dimensional asymptotically Euclidean Riemannian manifold.
A bounded domain \(E \subseteq M\) is a \emph{prehorizon domain} (with respect to the chosen asymptotically Euclidean end) if it is outer area minimizing, \(\Hm{n-3}(\partial E \setminus \rpartial E) = 0\), and \(\rpartial E\) is a smooth minimal hypersurface.
\end{definition}

\begin{definition}
Let \((M, g)\) be an asymptotically Euclidean Riemannian manifold.
The \emph{trapped region} of \((M, g)\) is the union of all prehorizon domains.
\end{definition}

The following lemma is an immediate consequence of the Schoen--Simon regularity theory \cite{SchoenSimon81} for stable stationary hypersurfaces, which we discuss briefly in \prettyref{sec:schoen-simon}.
\begin{lemma}\label{lem:prehorizon-regularity}
Let \((M, g)\) be an asymptotically Riemannian manifold of dimension \(n\).
If \(E\) is a prehorizon domain, then \(\Hm{\alpha}(\partial E \setminus \rpartial E) = 0\) if \(\alpha > n - 8\) and \(\alpha \geq 0\).
\end{lemma}
\begin{proof}
The minimal hypersurface \(\rpartial E\) is stable since \(E\) is outer area minimizing.
Moreover, \(\Hm{n-3}(\sing \rpartial E) = 0\) since \(E\) is a prehorizon domain.
Hence \prettyref{thm:minimal-surface-convergence} is applicable for the constant sequence \(\Sigma_k = \rpartial E\).
Since \(\Sigma_k \to \rpartial E\), it follows that \(\Hm{\alpha}(\partial E \setminus \rpartial E) = 0\) if \(\alpha > n - 8\) and \(\alpha \geq 0\).
\end{proof}

\section{The Solomon--White maximum principle}\label{sec:Solomon-White}
The central technical tool of this paper is the Solomon--White maximum principle \cite[Theorem,~p.~686]{SolomonWhite89}.
The version described in \cite[Additional~remarks,~pp.~690-691]{SolomonWhite89} (see also \cite[Theorem~4]{White10}) tells us the following:
\begin{theorem}[The Solomon--White maximum principle]\label{thm:original-solomon-white}
Let \(U\) be a smooth Riemannian manifold with boundary (not necessarily compact or complete).
Let \(T\) be a varifold in \(U\) which is stationary with respect to variations in \(U\).
If \(\partial U\) has positive mean curvature with respect to the outward-directed normal, then the support of \(T\) does not intersect \(\partial U\).
If \(\partial U\) is a connected minimal hypersurface, then the support of \(T\) contains \(\partial U\).
\end{theorem}
It is crucial for our application that the theorem is applicable to varifolds which are only stationary with respect to variations in \(U\) and not necessarily stationary with respect to variations in a larger manifold without boundary.
See \cite[Remark~(2),~p.~691]{SolomonWhite89} and the discussion after \cite[Theorem~1]{White10} for further comments on this point.
We will need a slightly stronger version of the theorem, which is not explicitly stated in \cite{SolomonWhite89}, but follows from the proof of \cite[Theorem,~p.~686]{SolomonWhite89}:
\begin{theorem}[A strengthened version of the Solomon--White maximum principle]\label{thm:strengthened-solomon-white}
Let \(U\) be a smooth Riemannian manifold with boundary (not necessarily compact or complete).
Suppose that \(\partial U\) has positive mean curvature with respect to the outward-directed normal.
Let \(T\) be a varifold in \(U\) and suppose that the support of \(T\) intersects \(\partial U\) at a point \(x\).
Then there is a variation which decreases the area of \(T\) to first order.
The variation can be chosen to be compactly supported in any neighborhood of \(x\).
Moreover, the normalized initial velocity of the variation can be made arbitrarily \(C^0\)-close to the inward-directed unit normal vector field of \(\partial U\).
\end{theorem}
We can see that the strengthened version of the theorem holds as follows:
The proof of the Solomon--White maximum principle given in \cite{SolomonWhite89} is performed in a manifold without boundary, and the role of \(\partial U\) is played by a smooth hypersurface \(M\).
The first step in the proof consists of choosing a point in \(M\), passing to a neighborhood of this point, and replacing \(M\) with a hypersurface with positive mean curvature which intersects the support of \(T\) only in the interior of the chosen neighborhood.
This is done by working in coordinates where \(M\) is the graph of a function \(u\), and replacing \(u\) by a function \(u_{s, \tau, \epsilon}\) with certain properties.
The second step consists of constructing a vector field orthogonal to the graph of \(u_{s, \tau, \epsilon}\) and proving that a variation of \(T\) by this vector field decreases area to first order.
It can be seen by tracing the proof that if \(\epsilon\) and \(s\) are sufficiently small, so that \(u_{s, \tau, \epsilon}\) is sufficiently close to \(u\), then the normalized variation vector field is \(C^0\)-close to the inward-directed unit normal vector field of \(\partial U\).
It can also be seen that we are free to choose \(\epsilon\) and \(s\) arbitrarily small without affecting the proof, since the only requirements for \(\epsilon\) and \(s\) are that they are sufficiently small compared to other quantities.
In other words, we may assume that the variation vector field is close to parallel to the inward-directed unit normal vector field of \(\partial U\).

\section{Convergence of stable stationary hypersurfaces}\label{sec:schoen-simon}
The second technical tool in this paper is the convergence theory for stable stationary hypersurfaces contained in the work of Schoen and Simon in \cite{SchoenSimon81}.
The result we need follows easily from \cite{SchoenSimon81}, but it is not explicitly stated there and we have not been able to find a proof of the exact result we need in the literature.
It is stated without proof in \cite[Theorem~1.3]{DeLellisTasnady13}, and a proof sketch can be found in \cite[Theorem~4.2]{DahlLarsson16}.
\begin{theorem}[Schoen--Simon \cite{SchoenSimon81}]\label{thm:minimal-surface-convergence}
Let \((M, g)\) be a Riemannian manifold of dimension \(n\) and let \(K \subset M\) be compact.
Let \((\Sigma_k)_{k = 1}^\infty\) be a sequence of smooth (but not necessarily closed) nonempty stable stationary hypersurfaces in \(K\).
Suppose that \(\Hm{n-3}(\sing(\Sigma_k)) = 0\) and \(\limsup_{k \to \infty} \Hm{n-1}(\Sigma_k) < \infty\).
Then there is a subsequence \((\Sigma_{k_i})_{i = 1}^\infty\) of \((\Sigma_k)_{k = 1}^\infty\) and a nonempty stable stationary hypersurface \(\Sigma_\infty \subset M\) such that
\begin{itemize}
	\item \(\Sigma_{k_i} \to \Sigma_\infty\) as varifolds,
	\item \(\Hm{\alpha}(\sing(\Sigma_\infty)) = 0\) if \(\alpha > n - 8\) and \(\alpha \geq 0\),
	\item for every open set \(\Omega\) with compact closure \(\overline{\Omega} \subseteq M \setminus \sing(\Sigma_\infty)\)
	\begin{itemize}
		\item \(\Hm{n-1}(\Sigma_\infty \cap \Omega) \leq \limsup_{i \to \infty} \Hm{n-1}(\Sigma_{k_i} \cap \Omega)\), and
		\item \((\Sigma_{k_i})_{i = 1}^\infty\) converges smoothly to \(\Sigma_\infty\) on \(\Omega\).
	\end{itemize}
\end{itemize}
\end{theorem}

\section{Proof of the main theorem}
The purpose of this section is to prove \prettyref{thm:main-theorem}, which is the main result of the paper.
The most important part of the proof is \prettyref{prop:prehorizon-set-union}, which tells us that the union of any two prehorizon domains is contained in a prehorizon domain.
This allows us to prove that there is a prehorizon domain which contains all other prehorizon domains, and that this largest prehorizon domain coincides with the trapped region.

We begin with an elementary observation which will be used in the proof of \prettyref{prop:prehorizon-set-union}.
\begin{lemma}\label{lem:smooth-factorization}
Let \(N\) be a smooth manifold and let \(I \subset \R\) be an open neighborhood of \(0\).
Let \(f \colon N \times I \to \R\) be a smooth function such that \(f(x, 0) = 0\) for all \(x \in N\).
Then there is a smooth function \(\phi \colon N \times I \to \R\) such that \(f(x, z) = \phi(x, z) z\) for all \((x, z) \in N \times I\).
\end{lemma}
\begin{proof}
It holds that
\[
f(x, z)
=
\int_0^1 \frac{d}{dt} f(x, tz) \, dt
=
\int_0^1 \frac{\partial f}{\partial z}(x, tz) \, z \, dt
=
\left(\int_0^1 \frac{\partial f}{\partial z}(x, tz) \, dt\right) z.
\]
Let
\[
\phi(x, z)
=
\int_0^1 \frac{\partial f}{\partial z}(x, tz) \, dt.
\]
Then \(\phi\) is a smooth function since \(f\) is smooth, and \(f(x, z) = \phi(x, z) z\).
\end{proof}

\begin{lemma}\label{lem:uniform-area-bound}
Let \((M, g)\) be an asymptotically Euclidean manifold.
Then there is a uniform bound for the perimeters of the prehorizon domains in \(M\).
\end{lemma}
\begin{proof}
The chosen asymptotically Euclidean end is foliated by spheres of positive mean curvature.
By the Solomon--White maximum principle, no prehorizon domain can contain points in this foliation.
Since prehorizon domains are outer area minimizing, their perimeters cannot be larger than the area of any sphere in the foliation.
\end{proof}

\begin{proposition}\label{prop:prehorizon-set-union}
The union of any two prehorizon domains is contained in a prehorizon domain.
\end{proposition}
\begin{proof}
Let \(E_1\) and \(E_2\) be prehorizon domains.
Let \(K\) be the region inside of a large coordinate sphere in the chosen asymptotically Euclidean end such that \(\partial K\) is a sphere of positive mean curvature and the region outside of \(K\) is foliated by spheres of positive mean curvature.
Consider the set of bounded domains which are contained in \(K\) and contain \(E_1 \cup E_2\).
By applying \prettyref{lem:Caccioppoli-compactness} and \prettyref{lem:perimeter-semicontinuity} to a sequence of such sets for which the perimeter converges to the infimum, we obtain a set of locally finite perimeter \(E\) which minimizes perimeter.
It is now sufficient to prove that \(E\) is a prehorizon domain.
By possibly replacing \(E\) with another representative of the same equivalence class we can make sure that \(\partial E = \overline{\rpartial E}\).
(See \cite[Proposition~3.1 and Theorem~4.4]{Giusti84}.)
By the Solomon--White maximum principle, \(\partial E\) cannot intersect \(\partial K\), and is hence contained in the interior of \(K\).
The complement of \(E\) is a neighborhood of the chosen asymptotically Euclidean end, and this complement has compact closure since \(E\) contains a bounded domain.
This proves that \(E\) itself is a bounded domain.
It holds that \(E\) is outer area minimizing: If some larger bounded domain had strictly smaller perimeter, then such a bounded domain \(E'\) would arise from the minimization problem defining \(E\), possibly with some larger set \(K'\) in place of \(K\).
However, \(E'\) is contained in \(K\) by the Solomon--White maximum principle since the region outside of \(K\) is foliated by spheres of positive mean curvature.
Hence \(E = E'\), proving that \(E\) is outer area minimizing.
We now only need to prove that \(\Hm{n-3}(\partial E \setminus \rpartial E) = 0\) and that \(\rpartial E\) is a smooth minimal hypersurface.
We will prove this locally.
Let \(A = A_1 \cup A_2 \cup A_3\) where
\[
A_1
=
\partial E_1 \setminus \rpartial E_1,
\]
\[
A_2
=
\partial E_2 \setminus \rpartial E_2,
\]
and \(A_3\) is the set of points \(y \in \rpartial E_1 \cap \rpartial E_2\) such that \(\rpartial E_1\) is tangent to \(\rpartial E_2\) at \(y\), but there is no neighborhood of \(y\) where \(\rpartial E_1\) and \(\rpartial E_2\) coincide.
We will now prove that every point in \(\partial E \setminus A\) has a neighborhood where \(\Hm{n-3}(\partial E \setminus \rpartial E) = 0\), \(\rpartial E\) is a smooth minimal hypersurface, and \(\partial E\) is stationary with respect to variations which are compactly supported in the neighborhood.
We do this in five cases:
\newcommand{\casefreedesc}{\(x \notin \partial E_1 \cup \partial E_2\)}
\newcommand{\casesmoothEonedesc}{\(x \in \rpartial E_1 \setminus \partial E_2\)}
\newcommand{\casesmoothEtwodesc}{\(x \in \rpartial E_2 \setminus \partial E_1\)}
\newcommand{\casecoincidencedesc}{\(x \in \rpartial E_1 \cap \rpartial E_2\) and there is a neighborhood of \(x\) where \(\rpartial E_1\) and \(\rpartial E_2\) coincide}
\newcommand{\casesmoothbothnontangentdesc}{\(x \in \rpartial E_1 \cap \rpartial E_2\) and \(\rpartial E_1\) is not tangent to \(\rpartial E_2\) at \(x\)}
\begin{itemize}
\item \casedef{case:free}{\casefreedesc}
\item \casedef{case:smooth-E-1}{\casesmoothEonedesc}
\item \casedef{case:smooth-E-2}{\casesmoothEtwodesc}
\item \casedef{case:coincidence}{\casecoincidencedesc}
\item \casedef{case:smooth-both-nontangent}{\casesmoothbothnontangentdesc}
\end{itemize}

\begin{proofcase}{\ref{case:free}}{\casefreedesc}
Pick a neighborhood \(U\) of \(x\) with closure disjoint from \(\partial E_1 \cup \partial E_2\).
In this neighborhood, \(\partial E\) is area minimizing by construction of \(E\), and hence it follows from the regularity theory for area minimizers (see for instance \cite[Theorem~8.4]{Giusti84},\cite[Theorem~11.8]{Giusti84}, \cite[Theorem~37.7]{Simon83}) that \(\Hm{n-3}(\partial E \setminus \rpartial E) = 0\), that \(\rpartial E\) is a smooth minimal hypersurface, and that \(\partial E\) is stationary with respect to variations which are compactly supported in \(U\).
\end{proofcase}

\begin{proofcase}{\ref{case:smooth-E-1}}{\casesmoothEonedesc}
By letting \(U\) be a sufficiently small neighborhood of \(x\), we can ensure that \(U \cap \partial E_2 = \emptyset\) and that the connected smooth hypersurface \(U \cap \rpartial E_1\) separates \(U\) into two components, one of which is \(U \setminus (E_1 \cup E_2 \cup \rpartial E_1)\).
Then it holds that \(N = (U \setminus (E_1 \cup E_2)) \cup \rpartial E_1\) is a smooth manifold with boundary \(\partial N = U \cap \rpartial E_1\).
Since \(E_1\) is a prehorizon domain, it holds that \(\partial N\) is a minimal hypersurface.
By the Solomon--White maximum principle it follows that \(U \cap \partial E\) contains \(U \cap \rpartial E_1\).
If it is possible to shrink \(U\) so that \(U \cap \partial E\) actually coincides with \(U \cap \rpartial E_1\), then we are done, since \(U \cap \rpartial E_1\) is a smooth minimal hypersurface.
If this were not possible, it would hold that \(x \in U \cap \overline{\partial E \setminus \rpartial E_1}\), and we can use an argument from \cite[Theorem~4]{White10} to obtain a contradiction:
Let \(W' = (U \cap \partial E) - (U \cap \rpartial E_1)\), where we view the two sets as unit density rectifiable varifolds.
Since \(U \cap \rpartial E_1\) is stationary and \(U \cap \partial E\) minimizes area to first order in the complement of \(E_1 \cup E_2\), it holds that \(W'\) minimizes area to first order in the complement of \(E_1 \cup E_2\).
Applying the Solomon--White maximum principle to \(W'\) in the manifold with boundary \(N\) we see that the support of \(W'\) contains \(U \cap \rpartial E_1\), which is a contradiction by definition of \(W'\).
Hence we may shrink \(U\) so that \(U \cap \partial E = U \cap \rpartial E_1\), proving that \(U \cap \partial E = U \cap \rpartial E\) is a smooth minimal hypersurface.
\end{proofcase}

\begin{proofcase}{\ref{case:smooth-E-2}}{\casesmoothEtwodesc}
This case is analogous to \prettyref{case:smooth-E-1}.
\end{proofcase}

\begin{proofcase}{\ref{case:coincidence}}{\casecoincidencedesc}
This case is analogous to \prettyref{case:smooth-E-1} and \prettyref{case:smooth-E-2}.
\end{proofcase}

\begin{proofcase}{\ref{case:smooth-both-nontangent}}{\casesmoothbothnontangentdesc}
We will prove that this case holds vacuously.
Suppose for contradiction that \(x \in \partial E \cap (\rpartial E_1 \cap \rpartial E_2)\) and that \(\rpartial E_1\) is not tangent to \(\rpartial E_2\) at \(x\).
Let \(U\) be a neighborhood of \(x\) such that it holds for \(i \in \{1,2\}\) that \(U \cap \rpartial E_i\) is connected, diffeomorphic to \(\R^{n-1}\), and separates \(U\) into two components.
Let \(\nu_i\) be the outward-directed unit normal vector field of \(U \cap \rpartial E_i\).
Let \((\nu_*)_x\) be the unit vector in direction \((\nu_1)_x + (\nu_2)_x\).
This is well-defined since \(\rpartial E_1\) is not tangent to \(\rpartial E_2\) at \(x\) so that \((\nu_1)_x + (\nu_2)_x \neq 0\).
Then \(g((\nu_*)_x, (\nu_i)_x) > 0\) for \(i \in \{1, 2\}\).

We need, in a neighborhood of \(x\), a hypersurface \(\Sigma_*\) with nonpositive mean curvature, with \(x \in \Sigma_* \subset \overline{E_1 \cup E_2}\), and with normal vector \((\nu_*)_x\) at \(x\).
The intersection of \(\rpartial E_1\) and \(\rpartial E_2\) at \(x\) is transverse, so \(I = U \cap \rpartial E_1 \cap \rpartial E_2\) is a smooth submanifold of codimension 2, after possibly shrinking \(U\).
Extend \(\nu_*\) by letting it be the unit vector field on \(I\) in direction \(\nu_1 + \nu_2\).
This is well-defined after possibly shrinking \(U\) so that \(\nu_1 + \nu_2 \neq 0\) on \(I\).
Let \(\Sigma_* \subset U\) be a hypersurface which contains \(I\), is orthogonal to \(\nu_*\) along \(I\), and has nonpositive mean curvature.
This exists since the mean curvature of \(I\) with respect to the normal vector field \(\nu_*\) can be compensated by the curvature in the direction orthogonal to \(I\) and \(\nu_*\).
After possibly shrinking \(U\) if necessary it holds that \(\Sigma_* \subset \overline{E_1 \cup E_2}\).
Extend \(\nu_*\) to the unit normal vector field on \(\Sigma_*\).

\prettyref{thm:strengthened-solomon-white}, the strengthened version of the Solomon--White maximum principle discussed in \prettyref{sec:Solomon-White}, now gives a vector field \(v\), supported in \(U\), which defines a variation which strictly decreases the perimeter of \(E\).
This vector field is outward-directed along \(\rpartial E_1\) and \(\rpartial E_2\) in some neighborhood of \(x\) since it can be chosen to be arbitrarily close to the normal vector field \(\nu_*\) of \(\Sigma_*\), which is outward-directed along \(\rpartial E_1\) and \(\rpartial E_2\) at \(x\).
Since \(E\) minimizes perimeter outside of \(E_1 \cup E_2\) and the variation along \(v\) decreases perimeter, we have a contradiction.
\end{proofcase}

We have proved that every point in \(\partial E \setminus A\) has a neighborhood where \(\Hm{n-3}(\partial E \setminus \rpartial E) = 0\), \(\rpartial E\) is a smooth minimal hypersurface, and \(\partial E\) is stationary with respect to variations which are compactly supported outside of \(A\).
We will now prove that the set \(A\) is small and compact, which is sufficient for the desired properties to hold on all of \(\partial E\).
Since \(E_1\) and \(E_2\) are prehorizon domains, it holds that \(\Hm{n-3}(A_1) = \Hm{n-3}(A_2) = 0\).
The sets \(A_1\) and \(A_2\) are compact since they are closed subsets of the compact sets \(\partial E_1\) and \(\partial E_2\).
Bounding the dimension of \(A_3\) is slightly more involved.
Consider a point \(x \in A_3\).
Choose a neighborhood \(U\) of \(x\), coordinates \(x_1, x_2, \ldots, x_n\) on \(U\), and a smooth function \(u \colon \R^{n-1} \to \R\) such that
\[
U \cap \rpartial E_1 = \{(x_1, \ldots, x_n) \colon x_n = 0\}
\]
\[
U \cap \rpartial E_2 = \{(x_1, \ldots, x_n) \colon x_n = u(x_1, \ldots, x_{n-1})\}.
\]
We may choose the coordinates to be normal coordinates along \(\rpartial E_1\), so that \(g^{nn} = 1\) and \(g^{ni} = 0\) for \(i \in \{1, \dots, n-1\}\).
The function \(u\) satisfies the minimal hypersurface equation
\[
\div \left( \frac{\grad(x_n - u)}{||\grad(x_n - u)||} \right)
=
0,
\]
in other words
\begin{equation}\label{eq:minimal-hypersurface}
\div \left( \frac{\grad(x_n)}{||\grad(x_n - u)||} \right)
-
\div \left( \frac{\grad(u)}{||\grad(x_n - u)||} \right)
=
0,
\end{equation}
where \(u\) and \(x_n\) are viewed as functions on \(U\).
We will now use this equation to construct a linear partial differential equation which is also satisfied by \(u\).

Concerning the first term in \eqref{eq:minimal-hypersurface}, note that
\[\begin{aligned}
\div \left( \frac{\grad(x_n)}{||\grad(x_n - u)||} \right)
&=
\div \left( \frac{||\grad(x_n)||}{||\grad(x_n - u)||} \frac{\grad(x_n)}{||\grad(x_n)||} \right)
\\
&=
\frac{||\grad(x_n)||}{||\grad(x_n - u)||} \div \left( \frac{\grad(x_n)}{||\grad(x_n)||} \right)
\\
& \qquad\qquad +
\frac{\grad(x_n)}{||\grad(x_n)||}\left( \frac{||\grad(x_n)||}{||\grad(x_n - u)||} \right)
\\
&=
\frac{1}{||\grad(x_n - u)||} \div \left( \frac{\grad(x_n)}{||\grad(x_n)||} \right) + 0.
\end{aligned}\]
The expression \(\div \left( \frac{\grad(x_n)}{||\grad(x_n)||} \right)\) gives the mean curvature of level surfaces of \(x_n\), and the level surface at level \(0\) is \(\rpartial E_1\) which has mean curvature zero.
Hence by \prettyref{lem:smooth-factorization}
\[
\frac{1}{||\grad(x_n - u)||} \div \left( \frac{\grad(x_n)}{||\grad(x_n)||} \right)
=
\phi x_n
\]
for some smooth function \(\phi\).
Evaluated at a point \((x, u(x))\), where \(x = (x_1, \dots, x_{n-1})\), this has the value
\[
\phi(x, u(x)) u(x).
\]

\newcommand{\sdetg}{\sqrt{\det g}}
The second term in \eqref{eq:minimal-hypersurface} can be written in coordinates, where \(\alpha\), \(\beta\), \(\gamma\), and \(\delta\) run through \(\{1, \dots, n\}\) and \(i\), \(j\), \(c\), and \(d\) run through \(\{1, \dots, n-1\}\), as
\[\begin{aligned}
\div \left( \frac{\grad(u)}{||\grad(x_n - u)||} \right)
&=
\frac{1}{\sdetg} \partial_\alpha \left( \sdetg \frac{g^{\alpha\beta} u_{,\beta}}{\sqrt{1 + g^{\gamma\delta} u_{,\gamma} u_{,\delta}}} \right)
\\
&=
\frac{1}{\sdetg} \partial_i \left( \sdetg \frac{g^{ij} u_{,j}}{\sqrt{1 + g^{cd} u_{,c} u_{,d}}} \right)
\\
&=
\frac{g^{ij}}{\sqrt{1 + g^{cd} u_{,c} u_{,d}}} u_{,ij}
+
\frac{1}{\sdetg} \partial_i \left(\frac{g^{ij} \sdetg}{\sqrt{1 + g^{cd} u_{,c} u_{,d}}} \right) u_{,j}
\end{aligned}\]
Introduce functions \(a^{ij}\) and \(b^j\) defined for \(x = (x_1, \dots x_{n-1})\) by
\[
a^{ij}(x)
=
-\frac{g^{ij}}{\sqrt{1 + g^{cd} u_{,c} u_{,d}}},
\]
\[
b^j(x)
=
-\frac{1}{\sdetg} \partial_i \left(\frac{g^{ij} \sdetg}{\sqrt{1 + g^{cd} u_{,c} u_{,d}}} \right),
\]
where the components of the metric are evaluated at the point \((x, u(x))\).
It now holds that \(v = u\) is a solution to the linear partial differential equation
\[
a^{ij}(x) v_{,ij}(x) + b^j(x) v_{,j}(x) + \phi(x, u(x)) v(x)
=
0.
\]
The equation is elliptic since \(g\) is a Riemannian metric.
Then it holds by \cite[Corollary~1.1]{HardtHoffmannOstenhofNadirashvili99} and \cite[Theorem~2]{Bar99} that the set of points \((x_1, \ldots, x_{n-1})\) such that \(u(x_1, \ldots, x_{n-1}) = 0\) and \(du(x_1, \ldots, x_{n-1}) = 0\) has Hausdorff dimension at most \(n-3\), after possibly shrinking \(U\).
Since the function \(u\) is smooth, this means that the set of points in \(U\) where \(\rpartial E_1\) is tangent to \(\rpartial E_2\) has Hausdorff dimension at most \(n-3\).
The set \(A_3\) is closed in \(M \setminus (A_1 \cup A_2)\) since it can be written as the difference between the closed set of points of tangency of \(\rpartial E_1\) and \(\rpartial E_2\), and the open set of points \(x\) where \(\rpartial E_1\) and \(\rpartial E_2\) coincide in some neighborhood of \(x\).

We have now proved that \(A\) is compact and \(\Hm{n-2}(A) = 0\).
Since \(E\) is outer area minimizing and \(\partial E\) is stationary with respect to all variations which are compactly supported outside of \(A\), it holds that \(\partial E\) satisfies the conditions of \prettyref{lem:extension-of-stationarity}, which tells us that \(\partial E\) is stationary with respect to all compactly supported variations.
Using this fact, we can prove that \(A_3\) is actually empty:
Suppose for contradiction that \(x \in A_3\).
Since \(\rpartial E_1\) is a minimal hypersurface, it follows from the Solomon--White maximum principle, \prettyref{thm:original-solomon-white}, that \(\partial E\) contains a neighborhood of \(x\) in \(\rpartial E_1\).
Analogously, \(\partial E\) contains a neighborhood of \(x\) in \(\rpartial E_2\).
Since \(E \supseteq E_1 \cup E_2\), this means that \(\rpartial E_1\) and \(\rpartial E_2\) coincide in a neighborhood of \(x\), which means that \(x \notin A_3\).
Hence \(A_3 = \emptyset\).
We have now proved that \(\Hm{n-3}(A) = \Hm{n-3}(A_1 \cup A_2) = 0\).
Hence \(\Hm{n-3}(\sing \partial E) = 0\), which is one of the conditions needed to apply the Schoen--Simon regularity theory for stable stationary hypersurfaces.
We know from the above argument that \(\partial E\) is stationary with respect to all compactly supported variations.
It is stable since it is outer area minimizing.
The regularity theory from \cite{SchoenSimon81} then tells us that \(\Hm{\alpha}(\partial E \setminus \rpartial E) = 0\) if \(\alpha > n - 8\) and \(\alpha \geq 0\).
This can be seen as a special case of \prettyref{thm:minimal-surface-convergence}.
Finally, \(\rpartial E\) is \(C^\infty\) by standard results on the regularity of \(C^2\) solutions to smooth elliptic partial differential equations.
This proves that \(E\) is a prehorizon domain.
\end{proof}

\begin{proposition}\label{prop:prehorizon-set-chain}
The union of a (possibly uncountable) chain of prehorizon domains is a prehorizon domain.
\end{proposition}
\begin{proof}
Consider a chain \(\{E_a\}_{a \in A}\) of prehorizon domains, indexed by the totally ordered set \(A\).
Then \(\bigcup_{a \in A} E_a\) is an open cover of itself.
Smooth manifolds are Lindelöf spaces, so this cover has a countable subcover \(\bigcup_{i = 1}^\infty E_{a_i}\).
By passing to a subsequence of \((a_i)_{i = 1}^\infty\), we may assume that this countable subcover is increasing.
The perimeters of the sets in the sequence are uniformly bounded by \prettyref{lem:uniform-area-bound}.
By the convergence theory for sets of locally finite perimeter, as stated in \prettyref{lem:Caccioppoli-compactness}, and the Schoen--Simon convergence theory of stable stationary hypersurfaces described by \prettyref{thm:minimal-surface-convergence}, it follows that a subsequence \((E_{a_i})_{i = 1}^\infty\) converges to a set \(E\) of locally finite perimeter such that \(\rpartial E\) is a smooth stable minimal hypersurface with \(\Hm{\alpha}(\partial E \setminus \rpartial E) = 0\) if \(\alpha > n - 8\) and \(\alpha \geq 0\).
Moreover, the perimeter of \(E\) does not exceed the limit inferior of the perimeters of the sets in the subsequence by \prettyref{lem:perimeter-semicontinuity}.
Since each \(E_{a_i}\) is outer area minimizing, it follows from this that \(E\) is outer area minimizing.
This proves that \(E\) is a prehorizon domain.
Since any subsequence of \((E_{a_i})_{i = 1}^\infty\) is an increasing cover of \(\bigcup_{a \in A} E_a\), it holds that \(E = \bigcup_{a \in A} E_a\).
\end{proof}

\begin{proposition}\label{prop:largest-prehorizon-domain}
There is a, necessarily unique, prehorizon domain which contains all other prehorizon domains.
\end{proposition}
\begin{proof}
Consider the partially ordered set of prehorizon domains ordered by inclusion.
By \prettyref{prop:prehorizon-set-chain}, every chain in this partially ordered set has an upper bound.
It follows from Zorn's lemma that the set has a maximal element.
Let \(E\) be such a maximal element.
If \(E'\) is any prehorizon domain then it follows from \prettyref{prop:prehorizon-set-union} that there is a prehorizon domain \(E''\) which contains \(E \cup E'\).
By maximality of \(E\), it holds that \(E'' = E\) so that \(E' \subseteq E\).
Hence \(E\) contains all other prehorizon domains.
\end{proof}

\begin{theorem}\label{thm:main-theorem}
Let \((M, g)\) be an \(n\)-dimensional asymptotically Euclidean Riemannian manifold with nonempty trapped region.
Suppose that \(n \geq 3\).
Then the trapped region is a prehorizon domain and contains all other prehorizon domains.
In particular, the boundary of the trapped region is a stable smooth minimal hypersurface except for a singular set of codimension at least \(8\).
\end{theorem}
\begin{proof}
Let \(E\) be the unique prehorizon domain which contains all other prehorizon domains, the existence of which is proved in \prettyref{prop:largest-prehorizon-domain}.
We will prove that \(\trappedregion = E\).
Since the trapped region is the union of all prehorizon domains, it holds that \(E \subseteq \trappedregion\).
For the reverse inclusion, let \(x \in \trappedregion\).
Since \(x \in \trappedregion\), there is a prehorizon domain \(E'\) containing \(x\).
Since \(E\) contains all prehorizon domains, it holds that \(E' \subseteq E\).
Hence \(x \in E' \subseteq E\) proving that \(\trappedregion \subseteq E\).
This means that \(\trappedregion = E\).
Hence \(\partial \trappedregion = \partial E = \overline{\rpartial E}\) is the closure of a smooth minimal hypersurface \(\rpartial E\) such that, by \prettyref{lem:prehorizon-regularity}, \(\Hm{\alpha}(\partial E \setminus \rpartial E) = 0\) if \(\alpha > n - 8\) and \(\alpha \geq 0\).
It is stable since \(E\) is a prehorizon domain and hence outer area minimizing.
This completes the proof.
\end{proof}

\bibliographystyle{amsalpha}
\bibliography{references}

\end{document}